%% file: AP14074.tex
\documentclass{aptpub}

\authornames{A. Bhaskar and Y. S. Song} 
\shorttitle{Multi-locus asymptotic sampling distributions} 

\usepackage{times}

\usepackage{verbatim}
\usepackage{amsmath}
\usepackage{amssymb}
\usepackage{latexsym}
\usepackage{tikz}
\usetikzlibrary{decorations,decorations.pathmorphing}
\usepackage{minitoc}
\usepackage{ifpdf}
\usepackage{subfigure}
\usepackage{cite}

\input macros.tex

\allowdisplaybreaks
\begin{document}
\title{Closed-form asymptotic sampling distributions\\ under
the coalescent with recombination for\\ an arbitrary
number of loci\\ }

\authorone[University of California, Berkeley]{Anand Bhaskar}
\authortwo[University of California, Berkeley]{Yun S. Song}

\addressone{Computer Science Division, University of California, Berkeley, CA 94720, USA.}
\addresstwo{Computer Science Division and Department of Statistics,  University of California, Berkeley, CA 94720, USA.}
\emailtwo{yss@stat.berkeley.edu}

\begin{abstract}
Obtaining a closed-form sampling distribution for the coalescent with recombination is a challenging problem.  In the case of \emph{two} loci, a new framework based on asymptotic series has recently  been developed to derive closed-form results when the recombination rate is moderate to large.  In this paper, an \emph{arbitrary} number of loci is considered and combinatorial approaches are employed to find closed-form expressions for the first couple of terms in an asymptotic expansion of the multi-locus sampling distribution.  These expressions are universal in the sense that their functional form in terms of the  marginal one-locus distributions applies to all finite- and infinite-alleles models of mutation.
\end{abstract}


\section{Introduction}
Coalescent processes, first introduced by Kingman \cite{kin:1982:JAP,kin:1982:SPA} about three decades ago, are widely-used stochastic models in population genetics that describe the genealogical ancestry of a sample of chromosomes randomly drawn from a population.  For many applications, the key quantity of interest  is the probability of observing the sample under a given coalescent model of evolution.  In the one-locus case with special models of mutation such as the infinite-alleles model or the finite-alleles parent-independent mutation  model, exact sampling distributions have been known in closed-form for many years \cite{wri:1949,ewe:1972}.  In contrast, for models with two or more loci with finite recombination rates, finding an exact, closed-form sampling distribution has remained a challenging open problem.  Therefore, most previous approaches have focused on Monte Carlo methods, including importance sampling \cite{gri:mar:1996,ste:don:2000, fea:don:2001, gri:etal:2008} and Markov chain Monte Carlo \cite{kuh:etal:2000, nie:2000, wan:ran:2008}.  Such methods have led to useful tools for population genetics analysis, but they are in general computationally intensive and their accuracy is difficult to characterize theoretically.

Recently, Jenkins and one of us \cite{jen:son:2009:G, jen:son:2010, jen:son:2011} made progress on the long-standing problem of finding sampling formulas for population genetic models with recombination
by proposing a new approach based on asymptotic expansion.  That work can be summarized as follows.  Consider an exchangeable random mating model with two loci, denoted $A$ and $B$. In the standard coalescent or diffusion limit, let $\tA$ and $\tB$  denote the respective population-scaled mutation rates at loci $A$ and $B$, and let $\rho$ denote the population-scaled recombination rate  between the two loci.  Given a sample configuration $\bfn$ (defined later in the text), assume that $\rho$ is large and consider an asymptotic expansion of the sampling probability $q(\bfn\mid\tA,\tB,\rho)$ in inverse powers of $\rho$:
\begin{equation}
q(\bfn\mid\tA,\tB,\rho) = q_0(\bfn\mid\tA,\tB) + \frac{q_1(\bfn\mid\tA,\tB)}{\rho} +
\frac{q_2(\bfn\mid\tA,\tB)}{\rho^2} + \cdots,
\label{eq:intro,expansion}
\end{equation}
where the coefficients $q_0(\bfn\mid\tA,\tB), q_1(\bfn\mid\tA,\tB), q_2(\bfn\mid\tA,\tB),\ldots,$ are independent of $\rho$.  The zeroth-order term $q_0$ corresponds to the sampling probability in the $\rho=\infty$ case (i.e., when the loci evolve independently), given simply by a product of marginal one-locus sampling probabilities \cite{eth:1979:JAP}.  
For either the infinite-alleles or an arbitrary finite-alleles model of mutation at each locus, 
Jenkins and Song \cite{jen:son:2009:G, jen:son:2010} derived a closed-form formula for the first-order term $q_1$ and showed that its functional form depends on the assumed model of mutation only through marginal one-locus sampling probabilities, a property which they termed \emph{universality}.
Further, they showed that the second-order term $q_2$ can be expressed as a sum of a closed-form formula plus another part that can be easily evaluated numerically using dynamic programming; they also showed that, for most sample configurations, the closed-form part of $q_2$ dominates the part that needs to be computed numerically.  More recently, the same authors \cite{jen:son:2011} utilized the diffusion process dual to the coalescent with recombination to develop a new computational technique for computing $q_k$ for all $k \geq 1$.  Moreover, they proved that only a finite number of terms in the asymptotic expansion is needed to recover (via the method of Pad\'e approximants) the \emph{exact} two-locus sampling probability as an analytic function of $\rho$ for all $\rho\in[0,\infty)$.  An immediate application of this work would be the composite-likelihood method \cite{hud:2001:G,mcv:etal:2002,mcv:etal:2004} for estimating fine-scale recombination rates which is based on combining two-locus sampling probabilities.

The main goal of this paper is to extend some of the mathematical results described above to more than two loci.  More precisely, we derive closed-form formulas for the first two terms ($q_0$ and $q_1$) in an asymptotic expansion (described later in detail) of the sampling distribution for an \emph{arbitrary} number of loci.  In general, the number of possible allelic combinations grows exponentially with the number of loci, and the system of equations that we need to solve 
is considerably more complex than that in the case of two loci.  
Note that the details of the computational techniques developed in \cite{jen:son:2009:G, jen:son:2010} are specific to the case of two loci, and new methods need to be developed to handle an arbitrary number of loci.
In this paper, we employ combinatorial approaches to make progress on the general case.
Our work shows that the universality property of $q_0$ and $q_1$ previously observed \cite{jen:son:2009:G, jen:son:2010} in the  two-locus case also applies to the case of an arbitrary number of loci.

The remainder of this paper is organized as follows.  In \sref{sec:prelim}, we introduce the multi-locus model to be considered in this paper and describe our notational convention.  
Our main results are summarized in \sref{sec:results} and an explicit example involving three loci is discussed in \sref{sec:example}.  
In \sref{sec:proofs}, we provide proofs of the main theoretical results presented in this paper.

\section{Preliminaries}
\label{sec:prelim}
Below we describe the model considered in this paper and lay out notation.  

\subsection{Model}

\begin{figure}
\centerline{\includegraphics{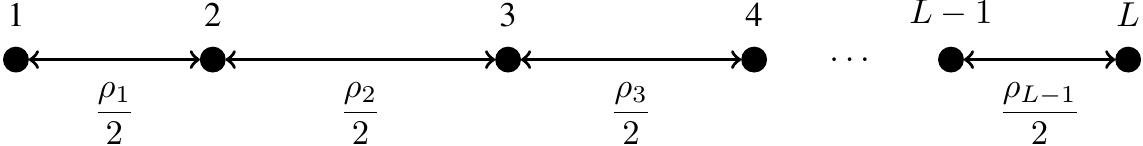}}
\caption{Illustration of $\nLoc$ loci arranged linearly. The population-scaled recombination rate between loci $l$ and $l+1$ is $\rho_l/2$.}
\label{fig:loci}
\end{figure}

We consider the diffusion limit of a neutral haploid exchangeable model of random mating with constant population size $2N$.  The haploid individuals in the population are referred to as gametes, and each gamete contains $\nLoc \geq 2$ loci labeled $1, 2, \dots, \nLoc$ and laid out linearly as illustrated in \fref{fig:loci}.  The probability of mutation at locus $l$ per gamete per generation is denoted by $u_l$, whereas the probability of recombination between loci $l$ and $l+1$ per gamete per generation is denoted by  $c_l$.  In the diffusion limit, as $N\to\infty$ we let $u_l\to 0$, for $1\leq l \leq \nLoc$, and $c_l\to 0$, for $1\leq l \leq \nLoc-1$, such that $4Nu_l\to\theta_l$ and $4Nc_l\to\rho_l$, where
$\theta_l$ and $\rho_l$ are population-scaled mutation and recombination rates, respectively.

Given a positive integer $k$, we use $[k]$ to denote the $k$-set $\{1,\ldots, k\}$.
At locus $l$, we assume that there are $K_l$ distinct possible allele types, labeled by $[K_l]$.  Mutation events at locus $l$ occur according to a Poisson process with rate $\theta_l/2$, and allelic changes are described by an ergodic Markov chain with transition matrix $\bfP{l} = (\P{l}_{ab})$; i.e., when a mutation occurs to an allele $a\in [K_l]$, it mutates to allele $b\in[K_l]$ with probability $\P{l}_{ab}$. The stationary distribution of $\bfP{l}$ is given by $\bfmath{\pi}^{(l)}$ with the $a$th entry $\pi^{(l)}_a$.

Recombination events between loci $l$ and $l+1$ occur at rate $\rho_l/2$.
In our work, we are interested in the case where $\rho_l\gg 1$, for all $l\in[\nLoc-1]$, and have similar order of magnitude.
Specifically, we re-express the recombination rates as $\rho_l = r_l \rho$, where $r_l$ are scaling constants, and consider an asymptotic expansion as $\rho \to \infty$.

\subsection{Notation}
As detailed later, the standard coalescent with recombination implies a closed system of recursion relations satisfied by sampling probabilities.  To obtain such a closed system of recursions, the allelic type space must be extended to allow gametes to be \emph{unspecified} at some loci.  We use the symbol $\un$ to denote an unspecified allele and define the $\nLoc$-locus \emph{haplotype} set $\cH$ as
\[
\cH = \bigg( ([K_1]\cup \{\un\}) \times \cdots \times ([K_\nLoc]\cup \{\un\}) \bigg)\setminus \{\un^\nLoc\}.
\]
Given a haplotype $h\in\cH$, we use $h_l\in[K_l]\cup\{\un\}$ to denote the allelic state of $h$ at locus $l$.
In what follows, we introduce definitions that are used throughout the paper.

The following two definitions explain how we denote samples:

\begin{definition}[$\n$ and $\e_h$, sample configurations]
A sample configuration is denoted by $\bfn=(n_h)_{h\in\cH}$, where $n_h$ is the number of times haplotype $h$ occurs in the sample, and 
the same letter $n$ in non-boldface is used to denote the total sample size of $\n$; i.e., $n=\sum_{h\in\cH} n_h$.
The symbol $\e_h$ is used to denote a sample configuration of size 1 for which $n_h=1$ and $n_{h'}=0$ for all $h'\neq h$.
Note that we can write $\n=\sum_{h\in\cH} n_h \e_h$. 
\end{definition}

\begin{definition}[$\marginal{\n}{l}$ and $\margVec{\n}$, marginal sample configurations]
Let $\bfn=(n_h)_{h\in\cH}$ be an $\nLoc$-locus sample configuration.
For $1\leq l\leq \nLoc$, the marginal sample size for locus $l$ and allele $a\in[K_l]$ is defined as $\marginal{n_a}{l} = \sum_{h\in\cH: h_l=a} n_h$, and the marginal sample size for locus $l$ is defined as $\marginal{n}{l} = \sum_{a \in [K_l]} \marginal{n_a}{l}$ (i.e., the total number of haplotypes with specified alleles at locus $l$).
Further, we use  $\marginal{\bfn}{l}=(\marginal{n_a}{l})_{a\in[K_l]}$ to denote the $K_l$-dimensional vector specifying the marginal sample configuration for locus $l$, and use
$\margVec{\n}=(\marginal{\n}{1},\ldots,\marginal{\n}{\nLoc})$ to denote the $\nLoc$-tuple of marginal sample configurations. 
Also note that if $h_l=\un$, then $\marginal{\e_h}{l}$ is a $K_l$-dimensional zero vector.
\end{definition}

The sets described in the following two definitions specify where mutations or recombinations can occur in a given haplotype:

\begin{definition}[$\sLoc{h}$, specified loci] 
For each locus $l$, the alleles labeled by $[K_l]$ are called \emph{specified} alleles.  Further, given a haplotype $h\in \cH$, we use $\sLoc{h} \subset [\nLoc]$ to denote the set of loci at which $h$ has specified alleles (i.e., not $\un$).
\end{definition}
    
\begin{definition}[$\bpts{h}$, break intervals] 
When considering recombination, a given haplotype $h\in\cH$ can be broken up between loci $l$ and $l+1$ if $\min(\sLoc{h}) \leq l < l+1 \leq \max(\sLoc{h})$.  
The index $l$ is used to refer to the break interval $(l,l+1)$, and the set of valid break intervals for haplotype $h$ is denoted by  $\bpts{h}=\{\min(\sLoc{h}),\ldots,\max(\sLoc{h})-1\}$.
\end{definition}

The two relations described below compare haplotypes.  When two haplotypes satisfy either relation, then their corresponding  lineages are allowed to coalesce.

\begin{definition}[$\compatible$, compatibility]
Given  a pair of haplotypes $h,h'\in\cH$,  if $h_l=h'_l$ for all $l\in \sLoc{h}\cap \sLoc{h'}$, then
we say that they are \emph{compatible} and write $h \compatible h'$.
\end{definition}

\begin{definition}[$\succeq$, containment] 
Given a pair of haplotypes $h,h'\in\cH$, we write $h\succeq h'$ if $\sLoc{h}\supseteq\sLoc{h'}$ and $h_l=h'_l$ for all $l\in \sLoc{h'}$.
\end{definition}

Corresponding to the types of event that may occur in the coalescent with recombination, we define the following operations on haplotypes:
\begin{enumerate}
\item (Mutate): Given a locus $l\in[\nLoc]$ and an allele $a \in [K_l]$, define $\sub{h}{l}{a}$ as the haplotype derived from $h\in\cH$ by substituting the allele at locus $l$ with $a$.
\item (Coalesce): If $h \compatible  h'$, define $C(h,h')$ as the haplotype $h''$ constructed as follows: 
\[
h''_l =
\left\{
\begin{array}{ll}
	 h_l,  & \text{if $h_l \neq \un$ and $h_l' = \un$},\\
	 h'_l, & \text{if $h_l = \un$ and $h'_l\neq \un$},\\
	 h_l=h'_l, & \text{otherwise}.
\end{array}
\right.
\]
\item (Break):
Given a break interval $l \in B(h)$, we use  $\rec{h}{l}{-} = (h_1,\ldots,h_l,\un,\ldots,\un)$ to denote the haplotype obtained from $h$ by replacing $h_j$ with $\un$ for all $j\geq l+1$, and $\rec{h}{l}{+} = (\un,\ldots,\un,h_{l+1},\ldots,h_\nLoc)$ to denote the haplotype obtained from $h$ by replacing $h_j$ with $\un$ for all $j\leq l$.
\end{enumerate}
Given a haplotype $h\in\cH$ and a set $X\subseteq [\nLoc]$, we define $\hapSet{h}{X}$ as the set of haplotypes that contain $h$ and are specified at the loci in $X$; i.e.,
\begin{equation}
\hapSet{h}{X} = \{ h'\in \cH \mid \sLoc{h'}\supseteq X \text{ and } h'\succeq h\}.
\label{eq:H(h,X)}
\end{equation}
Lastly, for a given subset $X\subset[\nLoc]$, we define $\rsum{X}$ as
\[
  \rsum{X}=r_{\min(X)} + r_{\min(X)+1} +\dots + r_{\max(X) - 1},
\]    
which corresponds to the total recombination rate (relative to $\rho$) between the first and the last loci in $X$.

\section{Main results on multi-locus asymptotic sampling distributions}
\label{sec:results}
For ease of notation, in most cases we suppress the dependence on the parameters $\{ r_l \}_{l=1}^{\nLoc-1}$ and $\{ \theta_l, \bfP{l} \}_{l=1}^\nLoc$
when writing sampling probabilities. By exchangeability, the probability of any \emph{ordered} configuration corresponding to sample $\bfn$ is invariant under all permutations of the sampling order.  Hence, we use $q(\bfn)$ without ambiguity to denote the stationary sampling probability of any \emph{particular} ordered configuration consistent with $\bfn$.
From the standard coalescent with recombination \cite{gri:1981, gol:1984, hud:1985, eth:gri:1990, gri:mar:1996}, one can derive a closed system of recursions and boundary conditions for which $q(\n)$ is the unique solution. Specifically, $q(\n)$ satisfies the following system of linear equations:
\begin{align}
\lefteqn{\sum_{h\in\cH} n_h \biggl[ (n - 1) + \sum_{l\in \sLoc{h}} \theta_l + \sum_{l\in \bpts{h}} \rho_l \biggr] q(\n)}\hspace{1.0cm}  & \notag \\
=&\sum_{h\in\cH} n_h\biggl[ (n_h - 1) q(\n - \e_h) + \sum_{h'\in\cH\st h\compatible h', h\neq h'}  n_{h'} q(\n - \e_{h} - \e_{h'} + \e_{C(h,h')}) \notag\\
{}& \phantom{\sum_{h\in\cH} n_h\biggl[}
+ \sum_{l\in\sLoc{h}} \theta_l \sum_{a \in [K_l]} P^{(l)}_{a,h_l} q(\n - \e_h + \e_{\sub{h}{l}{a}})\notag\\
{}& \phantom{\sum_{h\in\cH} n_h\biggl[}
+ \sum_{l\in \bpts{h}} \rho_{l} q(\n - \e_h + \e_{\rec{h}{l}{-}} + \e_{\rec{h}{l}{+}})
\biggr], \label{eq:linearsystem}
\end{align}
with boundary conditions 
\begin{align}\label{eq:boundaryconditions}
q(\e_h) = \prod_{l\in \sLoc{h}} \pi^{(l)}_{h_l}, \text{ for all $h\in\cH$.}
\end{align}
We define $q(\n) = 0$ if $n_h < 0$ for any $h \in \cH$.

The above closed system of equations is a full-rank linear system in the variables $q(\m)$ for all samples $\m$ reachable from the given sample $\n$ through repeated application of \eref{eq:linearsystem}.  Since $\rho_l=r_l \rho$, the entries of the matrix associated with the linear system are linear in $\rho$.  Hence, the entries in the inverse matrix are rational functions of $\rho$, thus implying that $q(\n)$ is a rational function of $\rho$, say $f(\rho)/g(\rho)$
where $f$ and $g$ are polynomials that depend on $\n$ and $r_l$.
Also, for every sample configuration $\n$, since $0 < q(\n) < 1$ as $\rho\to\infty$, $f$ and $g$ must be of the same degree in $\rho$.  Hence,
it follows that $q(\n)$ is also a rational function of $\rho^{-1}$ with both the numerator and the denominator having non-zero constant terms.
Hence, the Taylor series of $q(\n)$ about $\rho = \infty$ gives the following asymptotic expansion in inverse powers of $\rho$:
\begin{align} \label{eq:asf}
q(\n) = q_0(\n) + {q_1(\n) \over \rho} + {q_2(\n) \over \rho^2} + O\left({1 \over \rho^3}\right),
\end{align}
where the coefficients $q_0(\n),q_1(\bfn),q_2(\bfn),\ldots,$ are uniquely determined, and they depend on the sample configuration $\n$ and the model parameters $\{ \theta_l, \bfP{l} \}_{l=1}^\nLoc$ and $\{ r_l \}_{l=1}^{\nLoc-1}$, but not on $\rho$.
Note that $q_0(\n)$ corresponds to the sampling probability when $\rho$ is infinitely large,
in which case all haplotypes instantly break up into one-locus fragments and evolve independently back in time.   Hence, as proved by Ethier\cite{eth:1979:JAP} in the case of two loci, we expect $q_0(\n)$ to be given by the product of marginal one-locus  sampling probabilities.  The following result formalizes this intuition:

\begin{proposition}\label{prop:q0} 
For all $\nLoc$-locus sample configurations $\bfn$, 
\begin{equation}
 q_0(\n) = \prod_{l=1}^{\nLoc} \onelocq(\marginal{\n}{l} \mid \theta_l, \bfP{l}),
\label{eq:q0(n)}
\end{equation}
where $\onelocq$ denotes the  marginal one-locus sampling distribution.
\end{proposition}

\medskip
\noindent \emph{Remark:} An exact, closed-form expression for the one-locus sampling distribution $\onelocq(\marginal{\n}{l} \mid \theta_l, \bfP{l})$ is not known for general finite-alleles mutation models.
However, if a finite-alleles parent-independent mutation model is assumed at each locus (i.e., for each locus $l\in[\nLoc]$, the mutation transition matrix satisfies $\P{l}_{ab}=\pi^{(l)}_b$ for all $a,b\in[K_l]$), then in \eref{eq:q0(n)} one can use Wright's \cite{wri:1949} one-locus sampling formula
\[
\onelocq(\marginal{\n}{l} \mid \theta_l, \bfP{l}) = \frac{1}{(\theta)_{n^{(l)}}}\prod_{a=1}^{K_l} (\theta \pi^{(l)}_a)_{n^{(l)}_a},
\]
where $(x)_n = x(x+1)\cdots(x+n-1)$.  
\bigskip

A proof of \propref{prop:q0} is provided in \sref{sec:q0_prop}.
Since $q_0(\n)$ depends only on the marginal sample configuration $\margVec{\n} = (\marginal{\n}{1},\ldots,\marginal{\n}{\nLoc})$, henceforth we use $q_0(\n)$ and\break $q_0(\margVec{\n})$ interchangeably.

In \sref{sec:q1_lemma}, we apply the inclusion-exclusion principle to derive the following key result:

\begin{proposition}
\label{prop:q1recurrence}
The $q_1(\n)$ term in the asymptotic expansion \eqref{eq:asf} of $q(\n)$ is the unique solution to the recursion
\begin{align}
\lefteqn{\sum_{h\in\cH} n_h \sum_{l\in B(h)} r_l \biggl[ q_1(\n) - q_1(\n - \e_h + \e_{\rec{h}{l}{-}} + \e_{\rec{h}{l}{+}} )\biggr]}\hspace{1cm} & \notag\\
=& \displaystyle \sum_{h\in \cH\cup\{\un^\nLoc\}} \Biggl[ q_0(\margVec{\n} - \margVec{\e_h})\notag\\
{}& \hspace{1cm}
\times \sum_{\hXcond} (-1)^{|X-\sLoc{h}|} 
\Biggl(\sum_{h'\in\hapSet{h}{X}} n_{h'}\Biggr)\Biggl(\sum_{h''\in\hapSet{h}{X}} n_{h''}- 1\Biggr) \Biggr],
\label{eq:q1_simple_recurrence}
\end{align}
with boundary conditions
\begin{equation}
q_1(\e_h) = 0, \text{ for all $h\in\cH$.}
\label{eq:q1_boundary}
\end{equation}
We define $q_0(\margVec{\n}) = 0$ if $\marginal{n_a}{l} < 0$ for any $l \in [\nLoc]$ and $a \in [K_l]$.
\end{proposition}

\medskip
\noindent
In \sref{sec:q1_thm}, we prove that the closed-form expression  for $q_1(\n)$ in the following theorem is the unique solution to \eref{eq:q1_simple_recurrence} and \eref{eq:q1_boundary}:

\begin{theorem}\label{thm:q1}
Recursion \eqref{eq:q1_simple_recurrence} and boundary conditions \eqref{eq:q1_boundary} admit the following unique solution for $q_1(\n)$:
\begin{equation}
q_1(\n) = 
\displaystyle \sum_{h\in \cH\cup\{\un^\nLoc\}}  q_0(\margVec{\n} - \margVec{\e_h}) 
 \sum_{\hXcond}\frac{ (-1)^{|X-\sLoc{h}|}  }{\rsum{X}}
{\sum_{h'\in\hapSet{h}{X}} n_{h'}\choose 2},
\label{eq:thm}
\end{equation}
where $q_0$ is given by a product of marginal one-locus sampling distributions  as described in \propref{prop:q0}.
\end{theorem}
The intuition behind \propref{prop:q1recurrence} and \thmref{thm:q1} is as follows:  In \cite{jen:son:2010}, a formula for $q_1(\n)$ in the two-locus case was 
obtained by deriving a recursion satisfied by $q_1(\n)$
and by solving it using a probabilistic interpretation based on multivariate hypergeometric distributions.
The correct multi-locus generalization of the two-locus recursion for $q_1(\n)$ used in \cite{jen:son:2010} turns out to be the inclusion-exclusion type expression shown in \propref{prop:q1recurrence}, and an appropriate generalization of the associated probabilistic interpretation is based on Wallenius' noncentral hypergeometric distributions.
For ease of exposition, however, in \sref{sec:q1_thm} we provide a purely combinatorial proof of \thmref{thm:q1}.
For two loci, we show in \sref{sec:example} that the general multi-locus solution for $q_1(\n)$ in \eref{eq:thm} reduces to the solution found in \cite{jen:son:2010}.

In summary, \propref{prop:q0} and \thmref{thm:q1} imply the following  asymptotic expansion of the $L$-locus sampling distribution:


\begin{corollary}\label{qapproxtheorem}
For an arbitrary $\nLoc$-locus sample configuration $\n$, the sampling probability $q(\n)$ in the limit
		$\rho\to\infty$ has the following asymptotic expansion:
\begin{align*} 
q(\n) = & \prod_{l=1}^{\nLoc} \onelocq(\marginal{\n}{l})\\
{} & + {1 \over \rho} \sum_{h\in \cH\cup\{\un^\nLoc\}}  \left[\prod_{l=1}^{\nLoc} \onelocq(\marginal{\n}{l} - \marginal{\e_h}{l})\right]
\!
 \sum_{\hXcond}\!
\frac{(-1)^{|X-\sLoc{h}|} }{\rsum{X}}
{\sum_{h'\in\hapSet{h}{X}} n_{h'}\choose 2}
\\
{}& +  O\left({1 \over \rho^2}\right),
\end{align*}
where $\onelocq(\marginal{\n}{l})$ denotes the marginal one-locus sampling probability for locus $l$ with parameters $\theta_l$ and $\bfP{l}$.
\end{corollary}


Note that the formulas for $q_0(\n)$ and $q_1(\n)$ shown in \propref{prop:q0} and \thmref{thm:q1}, respectively, do not 
have any explicit dependence on the mutation parameters.  More precisely, the dependence on the assumed mutation model arises only implicitly  through the one-locus sampling probabilities $\onelocq(\marginal{\n}{l})$, and the formulas in \propref{prop:q0} and \thmref{thm:q1}  apply to all finite-alleles mutation models. 
In fact, by carrying out a similar line of derivation as that presented in this paper, it can be shown that 
the formulas in \propref{prop:q0} and \thmref{thm:q1} also apply to the case of the infinite-alleles model of mutation at each locus;  the marginal one-locus sampling probabilities $\onelocq(\marginal{\n}{l})$ in that case are given by the Ewens sampling formula \cite{ewe:1972}.
Jenkins and Song~\cite{jen:son:2009:G}  observed this universality property of $q_0$ and $q_1$ earlier in the case of two loci.  Our results imply that the universality property extends to an arbitrary number $\nLoc$ of loci.

\section{An explicit example: the three-locus case}
\label{sec:example}

Below we provide an explicit formula for $q_1(\n)$ in the case of $\nLoc=3$.
For ease of notation, we adopt the convention that the indices $i, j,$ and $k$ denote specified alleles which range over $[K_1], [K_2],$ and $[K_3]$, respectively.  Hence, $n_{ijk}$ denotes the number of fully specified haplotype $(i,j,k)$.
As in the rest of this paper, the symbol ``$\un$'' represents an unspecified allele.   Finally, the symbol ``$\any$'' for the index corresponding to locus $l$ represents a summation over all the alleles in $[K_l]$, while the symbol ``$\bullet$'' denotes a summation over $[K_l]\cup\{\un\}$.
  For example,  $n_{i\un \any} = \sum_{k\in[K_3]} n_{i \un k}$ and $n_{i\any\any} = \sum_{j\in[K_2]}\sum_{k\in[K_3]} n_{ijk}$, whereas $n_{i\bullet \any} = n_{i\un\any} + n_{i\any\any}$. In this notation, \thmref{thm:q1} implies that $q_1(\n)$ for $L=3$ is given by
\begin{align}
q_1(\n) = & \ 
  q_0(\margVec{\n}) \left[ {1 \over r_1} {n_{\any \any \bullet} \choose 2} + {1 \over r_2} {n_{\bullet \any \any} \choose 2} + {1 \over r_1 + r_2} {n_{\any \bullet \any} \choose 2}  - {1 \over r_1 + r_2} {n_{\any \any \any} \choose 2}\right] \nonumber\\
& + \sum_{i} q_0(\margVec{\n} - \margVec{\e_{i\un\un}}) \left[ {1 \over r_1 + r_2} {n_{i\any \any} \choose 2}  - {1 \over r_1 + r_2} {n_{i \bullet \any} \choose 2} - {1 \over r_1} {n_{i\any \bullet} \choose 2} \right]  \nonumber\\
& + \sum_{j} q_0(\margVec{\n} - \margVec{\e_{\un j \un}}) \left[ {1 \over r_1 + r_2} {n_{\any j\any}\choose 2} - {1 \over r_1} {n_{\any j \bullet} \choose 2}   - {1 \over r_2}{n_{\bullet j\any} \choose 2}\right]	\nonumber\\
& + \sum_{k} q_0(\margVec{\n} - \margVec{\e_{\un \un k}}) \left[ {1 \over r_1 + r_2} {n_{\any \any k}\choose 2} - {1 \over r_1 + r_2} {n_{\any \bullet k} \choose 2} - {1 \over r_2} {n_{\bullet \any k} \choose 2}\right]	\nonumber\\
& + \sum_{i,j} q_0(\margVec{\n} - \margVec{\e_{ij\un}}) \left[ {1 \over r_1}{n_{ij\bullet}\choose 2} - {1 \over r_1 + r_2} {n_{ij\any}\choose 2}\right]	    \nonumber\\
& + \sum_{j,k} q_0(\margVec{\n} - \margVec{\e_{\un jk}}) \left[ {1 \over r_2} {n_{\bullet jk}\choose 2}- {1 \over r_1 + r_2} {n_{\any jk}\choose 2}\right]\nonumber \\
& + \sum_{i,k} q_0(\margVec{\n} - \margVec{\e_{i\un k}}) \left[ {1 \over r_1 + r_2} {n_{i\bullet k}\choose 2} - {1 \over r_1 + r_2} {n_{i\any k}\choose 2} \right]     \nonumber\\
& + \sum_{i,j,k} q_0(\margVec{\n} - \margVec{\e_{ijk}}) {1 \over r_1 + r_2} {n_{ijk}\choose 2},
\label{eq:q1,L=3}
\end{align}
where $q_0$ is given by a product of marginal one-locus sampling probabilities.
If the sample does not contain any haplotype with an unspecified allele $\un$, \eref{eq:q1,L=3} reduces to the following:
\begin{align*}
q_1(\n) =
&  \left( {1 \over r_1} + {1 \over r_2}\right) q_0(\margVec{\n}){n_{\any \any \any}\choose 2}  
 - {1 \over r_1}  \sum_{i} q_0(\margVec{\n} - \margVec{\e_{i\un\un}}) {n_{i\any \any}\choose 2} \\
& - \left({1 \over r_1} + {1 \over r_2} - {1 \over r_1 + r_2}  \right)	
 \sum_{j} q_0(\margVec{\n} - \margVec{\e_{\un j \un}})  {n_{\any j\any}\choose 2} \\
& - {1 \over r_2} \sum_{k} q_0(\margVec{\n} - \margVec{\e_{\un \un k}}) { n_{\any \any k}\choose 2}\\
& + \left({1 \over r_1} - {1 \over r_1 + r_2} \right) \sum_{i,j} q_0(\margVec{\n} - \margVec{\e_{ij\un}}){n_{ij\any}\choose 2}  \\
& + \left({1 \over r_2} - {1 \over r_1 + r_2} \right) \sum_{j,k} q_0(\margVec{\n} - \margVec{\e_{\un jk}}){n_{\any jk}\choose 2}  \\
& + {1 \over r_1 + r_2} \sum_{i,j,k} q_0(\margVec{\n} - \margVec{\e_{ijk}})  {n_{ijk}\choose 2}.
\end{align*}
If the second locus is ignored, or equivalently, if every haplotype in the sample has an unspecified allele $\un$ at the second locus, \eref{eq:q1,L=3} becomes
\begin{align*}
q_1(\n) =\  & {1 \over r_1 + r_2} \Bigg[
  q_0(\margVec{\n})  {n_{\any \un \any} \choose 2} - \sum_{i} q_0(\margVec{\n} - \margVec{\e_{i\un\un}}) {n_{i \un \any} \choose 2} \\
&  \phantom{{1 \over r_1 + r_2} \Bigg[}
- \sum_{k} q_0(\margVec{\n} - \margVec{\e_{\un \un k}})  {n_{\any \un k} \choose 2}\\
& \phantom{{1 \over r_1 + r_2} \Bigg[}
+ \sum_{i,k} q_0(\margVec{\n} - \margVec{\e_{i\un k}}) {n_{i\un k} \choose 2} \Bigg],
\end{align*}
which coincides with the formula found by Jenkins and Song~\cite{jen:son:2009:G, jen:son:2010} in the case of $L=2$.

\section{Proofs of main results}
\label{sec:proofs}

In this section, we provide proofs of the results described in \sref{sec:results}.
For a given locus $l\in[\nLoc]$ and an allele $a\in[K_l]$, we use $\uvec{l}{a}$ to denote the $K_l$-dimensional unit vector where the $j$th component is 1 if $j=a$ and $0$ otherwise.

\subsection{Proof of \propref{prop:q0}} \label{sec:q0_prop}

By substituting the asymptotic expansion \eqref{eq:asf} into recursion \eqref{eq:linearsystem}, dividing by $\rho$, and letting $\rho \to \infty$, we obtain the following recursion for $q_0(\n)$:
\begin{align} \label{eq:q0recurrence}
\biggl[\sum_{h\in\cH} n_h \sum_{l\in \bpts{h}} r_l \biggr] q_0(\n)  = \sum_{h\in\cH} n_h \sum_{l\in \bpts{h}} r_{l} q_0(\n - \e_h + \e_{\rec{h}{l}{-}} + \e_{\rec{h}{l}{+}}).
\end{align}
We first establish the following lemma:

\begin{lemma}\label{lemma:q0} For every $\nLoc$-locus sample configuration $\n$,  $q_0(\bfn)$
	 depends only on the marginal sample configurations $\margVec{\n}=(\marginal{\bfn}{1},\ldots,\marginal{\bfn}{\nLoc})$:
\begin{equation}
 q_0(\n) = q_0(\margVec{\n}),
\label{eq:q0_dependence}
\end{equation}
where $\margVec{\n}$ is viewed as a sample configuration containing $\sum_{h\in\cH} n_h |\sLoc{h}|$ haplotypes each with a specified allele at exactly one locus and unspecified alleles elsewhere.
\end{lemma}

\begin{proof}
We use induction on the number of recombination events needed to transform a given sample configuration $\n$
into the configuration $\margVec{\n}$ that contains $\sum_{h\in\cH} n_h |\sLoc{h}|$ haplotypes each specified at exactly one locus.
The base case corresponds to a sample $\n$ consisting of haplotypes each specified at only one locus, in which case $\n = \margVec{\n}$, and \eref{eq:q0_dependence} is trivially true.
Given a sample configuration $\n$, consider the right hand side of \eref{eq:q0recurrence}. For any haplotype $h\in\cH$ satisfying $n_h>0$ and any $l\in B(h)$, let $\m = \n - \e_h + \e_{\rec{h}{l}{-}} + \e_{\rec{h}{l}{+}}$. The sample configuration $\m$ needs one less recombination event to be transformed to $\margVec{\m}$ than $\n$ needs to be transformed to $\margVec{\n}$. 
Hence, applying the induction hypothesis to $\m$, we have
$q_0(\m) = q_0(\margVec{\m})$. Noting that $\margVec{\m} = \margVec{\n}$ and using \eref{eq:q0recurrence}, we have
\[
\biggl[\sum_{h\in\cH} n_h \sum_{l\in \bpts{h}} r_l \biggr] q_0(\n)  = \sum_{h\in\cH} n_h \sum_{l\in \bpts{h}} r_{l} q_0(\margVec{\n}),
\]
which simplifies to \eref{eq:q0_dependence}.
\hfill{}\qed
\end{proof}

Now, let $\m$ denote a sample configuration such that $m_h=0$ for all $h$ with more than one specified locus (i.e., $|\sLoc{h}|> 1$).
Substituting the asymptotic expansion \eqref{eq:asf} for $\n=\m$ into \eqref{eq:linearsystem}, using \eref{eq:q0recurrence} to simplify, letting $\rho \to \infty$, and utilizing \lemmaref{lemma:q0}, we obtain the recursion
\begin{align}
\lefteqn{\sum_{l=1}^\nLoc \biggl[ \marginal{m}{l} (\marginal{m}{l} - 1) +  \theta_l \marginal{m}{l} \biggr] q_0(\marginal{\m}{1},\ldots,\marginal{\m}{\nLoc})=} \phantom{\hspace{0.5cm}} {}& \nonumber\\
{}& \sum_{l=1}^\nLoc \sum_{a \in [K_l]} \marginal{m_a}{l} (\marginal{m_a}{l} - 1) 
q_0(\marginal{\m}{1},\ldots,\marginal{\m}{l-1}, \marginal{\m}{l} - \uvec{l}{a},\marginal{\m}{l+1},\ldots,\marginal{\m}{\nLoc})  \nonumber\\
{}&+ \sum_{l=1}^{\nLoc} \theta_l \sum_{a, b \in [K_l]} P^{(l)}_{ab} \marginal{m_b}{l}
q_0(\marginal{\m}{1},\ldots,\marginal{\m}{l-1}, \marginal{\m}{l} - \uvec{l}{b} + \uvec{l}{a},\marginal{\m}{l+1},\ldots,\marginal{\m}{\nLoc}),
\label{eq:q0intermed}
\end{align}
and boundary conditions 
\begin{align}
q_0(\bfmath{0}^{(1)}, \ldots, \bfmath{0}^{(l-1)}, \uvec{l}{a_l}, \bfmath{0}^{(l+1)}, \ldots, \bfmath{0}^{(\nLoc)}) = \pi^{(l)}_{a_l}, \text{ for all $l \in [\nLoc]$ and $a_l \in [K_l]$},
\label{eq:q0_boundary}
\end{align}
where $\bfmath{0}^{(j)}$ denotes the $K_j$-dimensional zero vector.
Notice that recursion \eqref{eq:q0intermed} is the sum of $\nLoc$ one-locus recursions of the form
\begin{align*}
\biggl[ \marginal{m}{l} (\marginal{m}{l} - 1) + \theta_l \marginal{m}{l} \biggr] \onelocq(\marginal{\m}{l})  = & \sum_{a \in [K_l]} \marginal{m_a}{l} (\marginal{m_a}{l} - 1) \onelocq(\marginal{\m}{l} - \uvec{l}{a})  \\
{} & +
\theta_l \sum_{a, b \in [K_l]} P^{(l)}_{ab} \marginal{n_b}{l} \onelocq(\marginal{\m}{l} - \uvec{l}{b} + \uvec{l}{a}), \notag
\end{align*}
while boundary conditions \eref{eq:q0_boundary} are a product of one-locus boundary conditions $\onelocq(\uvec{l}{a}) = \pi^{(l)}_a$ and $\onelocq(\bfmath{0}^{(l)}) = 1$ for all $l\in[\nLoc]$ and $a\in[K_l]$.
Hence, it follows that $q_0(\marginal{\m}{1},\ldots,\marginal{\m}{\nLoc}) = \prod_{l=1}^{\nLoc} \onelocq(\marginal{\m}{l})$.  Finally, together with \lemmaref{lemma:q0}, letting $\marginal{\m}{l} = \marginal{\n}{l}$ for all $1\leq l \leq \nLoc$ in the above result implies $q_0(\n) = q_0(\marginal{\n}{1},\ldots,\marginal{\n}{\nLoc}) = \prod_{l=1}^\nLoc \onelocq(\marginal{\n}{l})$.
 \hfill{}\qed

\subsection{Proof of  \propref{prop:q1recurrence}}\label{sec:q1_lemma}
By an induction argument similar to that in the proof of \lemmaref{lemma:q0}, one can see that recursion \eref{eq:q1_simple_recurrence} and boundary conditions \eref{eq:q1_boundary} have a unique solution. 
Substituting \eqref{eq:asf} into both sides of \eqref{eq:linearsystem}, using \eref{eq:q0recurrence}, and taking the limit as $\rho \to \infty$, we obtain\\
\begin{align}
\lefteqn{\sum_{h\in\cH} n_h \biggl[ (n - 1) + \sum_{l\in \sLoc{h}} \theta_l\biggr] q_0(\n) +
\sum_{h\in\cH} n_h \sum_{l\in \bpts{h}} r_l  q_1(\n)} \hspace{1cm}  \notag \\
=&\sum_{h\in\cH} n_h\biggl[ (n_h - 1) q_0(\n - \e_h) + \sum_{h'\in\cH\st h\compatible h', h\neq h'}  n_{h'} q_0(\n - \e_{h} - \e_{h'} + \e_{C(h,h')}) \notag\\
{}& \phantom{\sum_{h\in\cH} n_h\biggl[} +
\sum_{l\in\sLoc{h}} \theta_l \sum_{a \in [K_l]} P^{(l)}_{a,h_l} q_0(\n - \e_h + \e_{\sub{h}{l}{a}}) \notag\\
{}& \phantom{\sum_{h\in\cH} n_h\biggl[} +
 \sum_{l\in \bpts{h}} r_{l} q_1(\n - \e_h + \e_{\rec{h}{l}{-}} + \e_{\rec{h}{l}{+}})
\biggr]. \label{eq:tmp_exp_q0} 
\end{align}
The terms that depend on mutation parameters can be eliminated by setting $\m = \margVec{\n}$ in \eqref{eq:q0intermed}, subtracting it from \eqref{eq:tmp_exp_q0}, and using the property of $q_0(\n)$ that it depends only on the marginal sample configuration at each locus.  As a consequence, the following simpler recursion can be obtained:
\begin{align}
\lefteqn{\sum_{h\in\cH} n_h \sum_{l\in B(h)} r_l \biggl[ q_1(\n) - q_1(\n - \e_h + \e_{\rec{h}{l}{-}} + \e_{\rec{h}{l}{+}} )\biggr]} \hspace{1cm}  
& \nonumber \\
=& 
 \sum_{h,h'\in\cH\st h\compatible h', h\neq h'} n_h n_{h'} q_0(\n - \e_{h} - \e_{h'} + \e_{C(h,h')})
\nonumber\\
{}& +  \sum_{h\in\cH} n_h (n_h - 1) q_0(\n - \e_h)
- \Bigl[ n (n - 1) - \sum_{l=1}^{\nLoc} \marginal{n}{l} (\marginal{n}{l} - 1) \Bigr]q_0(\n)\nonumber \\
{}&-\sum_{l=1}^\nLoc \sum_{a \in [K_l]} \marginal{n_a}{l} (\marginal{n_a}{l} - 1) 
q_0(\marginal{\n}{1},\ldots,\marginal{\n}{l-1}, \marginal{\n}{l} - \uvec{l}{a},\marginal{\n}{l+1},\ldots,\marginal{\n}{\nLoc}).
\label{eq:q1_init_recurrence} 
\end{align}
We also have the boundary conditions $q_1(\e_h) = 0$ for all $h\in\cH$ since $q(\e_h) = q_0(\e_h)$.

As the left hand side and boundary conditions of \eref{eq:q1_init_recurrence} are identical to that of \eref{eq:q1_simple_recurrence}, it suffices to establish that their right hand sides are also identical to show equivalence. 
Note that the right hand side of \eqref{eq:q1_simple_recurrence} can be written as follows:
\begin{align}
&{\ds \sum_{h\in \cH\cup\{\un^\nLoc\}}  q_0(\margVec{\n} - \margVec{\e_h}) \sum_{\hXcond} (-1)^{|X-\sLoc{h}|} \Biggl(\sum_{h'\in\hapSet{h}{X}}\! n_{h'}\Biggr)\Biggl(\sum_{h''\in\hapSet{h}{X}} \!\! n_{h''}- 1\Biggr) } 	\notag  \\
& =  \ds \sum_{h\in \cH\cup\{\un^\nLoc\}}  q_0(\margVec{\n} - \margVec{\e_h}) \Biggl[ \sum_{\hXcond} (-1)^{|X-\sLoc{h}|} \Biggl(\sum_{h',h''\in\hapSet{h}{X}} n_{h'} n_{h''} \Biggr) \Biggr]    \notag    \\
&\hspace{5mm} - \ds \sum_{h\in \cH\cup\{\un^\nLoc\}}  q_0(\margVec{\n} - \margVec{\e_h}) \Biggl[ \sum_{\hXcond} (-1)^{|X-\sLoc{h}|} \Biggl(\sum_{h'\in\hapSet{h}{X}} n_{h'} \Biggr) \Biggr].  \label{eq:q1recur_breakup}
\end{align}
The first term on the right hand side of \eref{eq:q1recur_breakup} can be rewritten as
\begin{align}
\lefteqn{\ds \sum_{h\in \cH\cup\{\un^\nLoc\}}  q_0(\margVec{\n} - \margVec{\e_h}) \Biggl[ \sum_{\hXcond} (-1)^{|X-\sLoc{h}|} \Biggl(\sum_{h',h''\in\hapSet{h}{X}} n_{h'} n_{h''} \Biggr) \Biggr]} \hspace{0.2cm} &    \nonumber\\
& = \sum_{h\in \cH\cup\{\un^\nLoc\}}  q_0(\margVec{\n} - \margVec{\e_h}) \Biggl[\sum_{h',h'' \succeq h} n_{h'} n_{h''} \Biggl( \sum_{\substack{X: |X| \geq 2, \\ \sLoc{h} \subseteq X \subseteq (\sLoc{h'} \cap \sLoc{h''})}} \!\! (-1)^{|X-\sLoc{h}|} \Biggr) \Biggr]	\nonumber\\
& = \sum_{h\in \cH\cup\{\un^\nLoc\}}  q_0(\margVec{\n} - \margVec{\e_h}) \Bigg\{ \sum_{h',h'' \succeq h} n_{h'} n_{h''} 
\Biggl[  \sum_{X: \sLoc{h} \subseteq X \subseteq (\sLoc{h'} \cap \sLoc{h''})} \!\! (-1)^{|X-\sLoc{h}|}  \nonumber\\
& \hspace{1cm} - \sum_{\substack{X: |X| = 0, \\ \sLoc{h} \subseteq X \subseteq (\sLoc{h'} \cap \sLoc{h''})}}\!\! (-1)^{|X-\sLoc{h}|} \hspace{2mm}
- \sum_{\substack{X: |X| = 1, \\ \sLoc{h} \subseteq X \subseteq (\sLoc{h'} \cap \sLoc{h''})}}\!\! (-1)^{|X-\sLoc{h}|}   \Biggr] \Bigg\},
\label{eq:tmp_first_sum}
\end{align}
where the first equality follows because, by definition \eref{eq:H(h,X)}, the condition that $h' \in \hapSet{h}{X}$ is equivalent to $h' \succeq h$ and $X\subset\sLoc{h'}$, and similarly for $h''$.
Now, by the inclusion-exclusion principle,
\[
\sum_{X: \sLoc{h} \subseteq X \subseteq (\sLoc{h'} \cap \sLoc{h''})} (-1)^{|X-\sLoc{h}|} = \delta_{\sLoc{h'} \cap \sLoc{h''}, \sLoc{h}},
\]
where for any sets $A$ and $B$, $\delta_{A,B} = 1$ if $A = B$, and $\delta_{A,B} = 0$ otherwise.
Then the second to last line of \eref{eq:tmp_first_sum} simplifies to
\begin{align*}
\lefteqn{\sum_{h\in \cH\cup\{\un^\nLoc\}}  q_0(\margVec{\n} - \margVec{\e_h}) \Biggl[ \sum_{h',h'' \succeq h} n_{h'} n_{h''} \Biggl( \sum_{X: \sLoc{h} \subseteq X \subseteq (\sLoc{h'} \cap \sLoc{h''})} (-1)^{|X-\sLoc{h}|} \Biggr) \Biggr]} \hspace{1cm} \\
& = \sum_{h\in \cH\cup\{\un^\nLoc\}}  q_0(\margVec{\n} - \margVec{\e_h}) \Biggl[ \sum_{h',h'' \succeq h} n_{h'} n_{h''} \delta_{\sLoc{h'} \cap \sLoc{h''}, \sLoc{h}} \Biggr]   \\
& = \sum_{h', h'': h' \compatible h''} n_{h'} n_{h''}  q_0(\margVec{\n} - \margVec{\e_{h'}} - \margVec{\e_{h''}} + \margVec{\e_{C(h',h'')}})    \\
& = \sum_{h', h'': h' \compatible h''} n_{h'} n_{h''}  q_0(\n - \e_{h'} - \e_{h''} + \e_{C(h',h'')})    \\
& = \sum_{h', h'': h' \compatible h'', h' \neq h''} n_{h'} n_{h''}  q_0(\n - \e_{h'} - \e_{h''} + \e_{C(h',h'')}) + \sum_{h \in \cH} n_h^2 q_0(\n - \e_{h}),
\end{align*}
where the second equality follows because $S(h) = S(h') \cap S(h'')$ and $h', h'' \succeq h$ imply that $h'$ and $h''$ are compatible by definition, and hence $\margVec{\e_h} = \margVec{\e_{h'}} + \margVec{\e_{h''}} - \margVec{\e_{C(h',h'')}}$. The third equality holds because $q_0$ depends only on the marginal sample configurations, and the last equality follows because when $h' = h'' = h$, $C(h', h'') = h$.
The terms in the last line of  \eref{eq:tmp_first_sum} can be simplified as follows:
\[
\sum_{h\in \cH\cup\{\un^\nLoc\}}  q_0(\margVec{\n} - \margVec{\e_h}) \Biggl[ \sum_{h',h'' \succeq h} n_{h'} n_{h''}\!\! \sum_{\substack{X: |X| = 0, \\ \sLoc{h} \subseteq X \subseteq (\sLoc{h'} \cap \sLoc{h''})}}\!\! (-1)^{|X-\sLoc{h}|} 
\Biggr] = n^2 q_0(\n),
\]
since the only $h, h', h''$ and $X$ that satisfy the conditions of the inner summation over $X$ are $h=\{\un^\nLoc\}$, $h', h'' \in \cH$ and $X = \varnothing$.
Furthermore, one can also see that
\begin{align*}
\lefteqn{\sum_{h\in \cH\cup\{\un^\nLoc\}}  q_0(\margVec{\n} - \margVec{\e_h}) \Biggl[ \sum_{h',h'' \succeq h} n_{h'} n_{h''} \Biggl( \sum_{\substack{X: |X| = 1, \\ \sLoc{h} \subseteq X \subseteq (\sLoc{h'} \cap \sLoc{h''})}} (-1)^{|X-\sLoc{h}|} \Biggr) \Biggr]} \hspace{0.5cm} &\\
& = 
\sum_{l=1}^{\nLoc} \sum_{a \in [K_l]} (\marginal{n_a}{l})^2 q_0(\marginal{\n}{1},\ldots,\marginal{\n}{l-1}, \marginal{\n}{l} - \uvec{l}{a},\marginal{\n}{l+1},\ldots,\marginal{\n}{\nLoc}) - 
 \sum_{l=1}^{\nLoc} (\marginal{n}{l})^2  q_0(\n),
\end{align*}
since the only $h, h', h''$ and $X$ that satisfy the conditions of the summation over $X$ are:
\begin{itemize}
\item $X = \{ l \}$ for some $l \in [\nLoc]$ and $h=\{\un^\nLoc\}$. Because of the condition that $X \subseteq (\sLoc{h'} \cap \sLoc{h''})$, $h'$ and $h''$ range over all haplotypes that are specified at locus $l$.
\item $X = \{ l \}$ for some $l \in [\nLoc]$ and $h$ such that $h_l = a$ for some $a \in [K_l]$ and $h_{l'} = \un$ for all $l' \neq l$. Because $h', h'' \succeq h$ and $X \subseteq (\sLoc{h'} \cap \sLoc{h''})$, $h'$ and $h''$ range over all haplotypes with allele $a$ at locus $l$.
\end{itemize}

In summary, \eref{eq:tmp_first_sum}, which corresponds to the first term on the right hand side of  \eref{eq:q1recur_breakup}, can be written as
\begin{align}
\lefteqn{\sum_{h\in \cH\cup\{\un^\nLoc\}}  q_0(\margVec{\n} - \margVec{\e_h}) \Biggl[ \sum_{\hXcond} (-1)^{|X-\sLoc{h}|} \Biggl(\sum_{h',h''\in\hapSet{h}{X}} n_{h'} n_{h''} \Biggr) \Biggr]} \hspace{1cm}    \notag \\
=& \sum_{h', h'': h' \compatible h'', h' \neq h''} n_{h'} n_{h''}  q_0(\n - \e_{h'} + \e_{h''} - \e_{C(h',h'')}) + \sum_{h \in \cH} n_h^2 q_0(\n - \e_{h})  \notag \\
{}&  - [n^2 - \sum_{l=1}^{\nLoc} (\marginal{n}{l})^2] q_0(\n)\notag\\
{}& - \sum_{l=1}^{\nLoc} \sum_{a \in [K_l]} (\marginal{n_a}{l})^2 q_0(\marginal{\n}{1},\ldots,\marginal{\n}{l-1}, \marginal{\n}{l} - \uvec{l}{a},\marginal{\n}{l+1},\ldots,\marginal{\n}{\nLoc}). \label{eq:first_term}
\end{align}
By following similar steps as above, the second term on the right hand side of \eref{eq:q1recur_breakup} can be written as
\begin{align}
\lefteqn{\sum_{h\in \cH\cup\{\un^\nLoc\}}  q_0(\margVec{\n} - \margVec{\e_h}) \Biggl[ \sum_{\hXcond} (-1)^{|X-\sLoc{h}|} \Biggl(\sum_{h'\in\hapSet{h}{X}} n_{h'} \Biggr) \Biggr]} \hspace{0.5cm} &    \notag \\    
= &  \sum_{h\in \cH\cup\{\un^\nLoc\}}  q_0(\margVec{\n} - \margVec{\e_h}) \Biggl[\sum_{h' \succeq h} n_{h'} 
\Biggl( \sum_{X: \sLoc{h} \subseteq X \subseteq \sLoc{h'}} (-1)^{|X-\sLoc{h}|} \notag  \\
{} &\hspace{2cm} - \sum_{\substack{X: |X| = 0, \\ \sLoc{h} \subseteq X \subseteq \sLoc{h'}}} (-1)^{|X-\sLoc{h}|} 
 - \sum_{\substack{X: |X| = 1, \\ \sLoc{h} \subseteq X \subseteq (\sLoc{h'} \cap \sLoc{h''})}} (-1)^{|X-\sLoc{h}|} \Biggr) \Biggr]  \notag \\
= &  \sum_{h\in \cH\cup\{\un^\nLoc\}}  q_0(\margVec{\n} - \margVec{\e_h}) \Biggl[\sum_{h' \succeq h} n_{h'} 
\Biggl( \delta_{S(h'),S(h)}  \notag \\
{} &\hspace{2cm}- \sum_{\substack{X: |X| = 0, \\ \sLoc{h} \subseteq X \subseteq \sLoc{h'}}} (-1)^{|X-\sLoc{h}|} 
 - \sum_{\substack{X: |X| = 1, \\ \sLoc{h} \subseteq X \subseteq (\sLoc{h'} \cap \sLoc{h''})}} (-1)^{|X-\sLoc{h}|} \Biggr) \Biggr]  \notag \\
=&  \sum_{h \in \cH} n_h q_0(\n - \e_h) - n q_0(\n) + \sum_{l=1}^{\nLoc} \marginal{n}{l} q_0(\n) \nonumber \\
{} & - \sum_{l=1}^{\nLoc} \sum_{a \in [K_l]} \marginal{n_a}{l} q_0(\marginal{\n}{1},\ldots,\marginal{\n}{l-1}, \marginal{\n}{l} - \uvec{l}{a},\marginal{\n}{l+1},\ldots,\marginal{\n}{\nLoc}),
\label{eq:second_term}
\end{align}
where the second equality follows from the inclusion-exclusion principle.
Finally, subtracting \eref{eq:second_term} from \eref{eq:first_term}, we see that the right hand side of \eref{eq:q1_simple_recurrence} is equal to the right hand side of \eref{eq:q1_init_recurrence}. 
\hfill\qed


\subsection{Proof of \thmref{thm:q1}} \label{sec:q1_thm}

We first show that the boundary conditions \eref{eq:q1_boundary} are satisfied by \eref{eq:thm}: 
If $\n = \e_g$ for some $g \in \cH$, then in the right hand side of \eref{eq:thm}, the only $h$ that can potentially contribute to the summation 
are those satisfying $g \succeq h$, since otherwise $q_0(\margVec{\e_g}-\margVec{\e_h})=0$.  However, for $g \succeq h$, if $\sLoc{h} \subseteq X \subseteq \sLoc{g}$, then $\sum_{h'' \in \hapSet{h}{X}} n_{h''} - 1 = 0$ since $g \in \hapSet{h}{X}$, and if $X \not\subset \sLoc{g}$, then $\sum_{h' \in \hapSet{h}{X}} n_{h'} = 0$ since $g \notin \hapSet{h}{X}$. Therefore, in the right hand side of  \eref{eq:thm}, 
\[
{\sum_{h'\in\hapSet{h}{X}} n_{h'}\choose 2} = \frac{1}{2} \Bigg(\sum_{h' \in \hapSet{h}{X}} n_{h'}\Bigg)\Bigg(\sum_{h'' \in \hapSet{h}{X}} n_{h''} - 1\Bigg) = 0,
\]
and so $q_1(\e_g) = 0$ for all $g \in \cH$.

We now show that the recursion \eref{eq:q1_simple_recurrence} is satisfied by \eref{eq:thm}.
Substituting \eref{eq:thm} into the left hand side of \eref{eq:q1_simple_recurrence}, we obtain
\begin{align*}
\lefteqn{\sum_{g\in\cH} n_g \sum_{l\in B(g)} r_l \biggl[ q_1(\n) - q_1(\n - \e_g + \e_{\rec{g}{l}{-}} + \e_{\rec{g}{l}{+}} )\biggr]} \hspace{1cm}& \\
= & \sum_{g\in\cH} n_g \sum_{l\in B(g)} r_l 
\sum_{h\in \cH\cup\{\un^\nLoc\}}  q_0(\margVec{\n} - \margVec{\e_h}) 
\Bigg\{  \sum_{\hXcond} {(-1)^{|X-\sLoc{h}|} \over \rsum{X}} \\
  & \hspace{0.5cm}  \times \left[ - \left(\sum_{h'\in\hapSet{h}{X}} n_{h'}\right) \left(\sum_{h''\in\hapSet{h}{X}} \delta_{h'',\rec{g}{l}{-}} + \delta_{h'',\rec{g}{l}{+}} - \delta_{h'',g} \right) \right. \\
  & \left. \left. \phantom{\hspace{0.5cm}  \times \left[ \right.} 
- {\sum_{h'\in\hapSet{h}{X}} \delta_{h',\rec{g}{l}{-}} + \delta_{h',\rec{g}{l}{+}} - \delta_{h',g} \choose 2} \right]
\right\},
\end{align*}
where for any haplotypes $g$ and $h$, $\delta_{g,h} = 1$ if $g = h$ and $\delta_{g,h} = 0$ otherwise. Also, in the previous line, by ${x \choose 2}$ we mean $x(x-1)/2$ for all $x\in\bbR$.
Note that if $g \notin \hapSet{h}{X}$, then $\rec{g}{l}{-}, \rec{g}{l}{+} \notin \hapSet{h}{X}$, and so 
\[
\sum_{h' \in \hapSet{h}{X}} \delta_{h',\rec{g}{l}{-}} + \delta_{h',\rec{g}{l}{+}} - \delta_{h',g} = 0.
\]
Interchanging the summation over $g, l$ and $X$, and introducing the restriction that $g \in \hapSet{h}{X}$ (i.e., $S(g) \supseteq X$ and $g \succeq h$), we obtain
\begin{align}
\lefteqn{\sum_{g\in\cH} n_g \sum_{l\in B(g)} r_l \biggl[ q_1(\n) - q_1(\n - \e_g + \e_{\rec{g}{l}{-}} + \e_{\rec{g}{l}{+}} )\biggr]}\hspace{1cm} & \notag \\    
    = &
    \sum_{h\in \cH\cup\{\un^\nLoc\}}  q_0(\margVec{\n} - \margVec{\e_h}) \left\{ \rule{0mm}{0.7cm}\right. \sum_{\hXcond} {(-1)^{|X-\sLoc{h}|} \over \rsum{X}} \sum_{g \in \hapSet{h}{X}} n_g \sum_{l\in B(g)} r_l \notag \\
      & \hspace{0.2cm}\times  \left[ - \left(\sum_{h'\in\hapSet{h}{X}} n_{h'}\right) \left(\sum_{h''\in\hapSet{h}{X}} \delta_{h'',\rec{g}{l}{-}} + \delta_{h'',\rec{g}{l}{+}} - \delta_{h'',g} \right) \right. \nonumber\\
      & \phantom{\hspace{0.2cm}  \times \left[\right.} \left. \left. \rule{0mm}{0.7cm}
 - {\sum_{h'\in\hapSet{h}{X}} \delta_{h',\rec{g}{l}{-}} + \delta_{h',\rec{g}{l}{+}} - \delta_{h',g} \choose 2}
\right] \right\}. 
\label{eq:q1_tmp1} 
\end{align}
Now, for $g \in \hapSet{h}{X}$, note that 
\begin{equation}
	\sum_{h' \in \hapSet{h}{X}} \delta_{h', g} = 1.
\label{eq:sum1}	
\end{equation}
We utilize this identity in the ensuing discussion.
There are three cases for the recombination break interval $l \in B(g)$ in the right hand side of \eref{eq:q1_tmp1}:

\begin{figure}
\centering
\subfigure[]{\label{fig:case1}\includegraphics[width=0.42\textwidth]{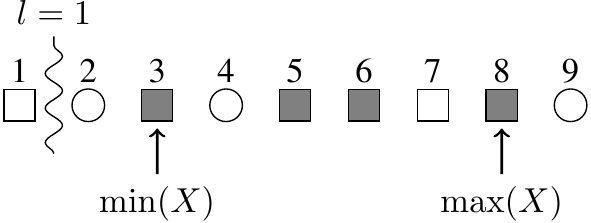}}
 \hspace{1.5cm}
\subfigure[]{\label{fig:case3}\includegraphics[width=0.42\textwidth]{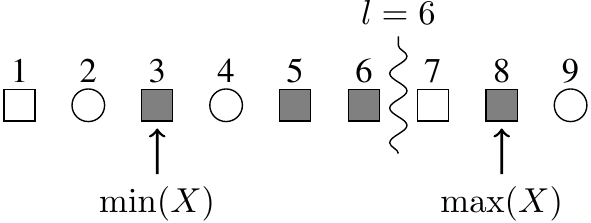}}
\caption{
Illustration of sub-cases considered in the proof of \thmref{thm:q1}.  Here $X\subseteq S(g)$, where $g\succeq h$.
Squares denote the loci in $\sLoc{g}$, while shaded squares denote the loci in $X$ (and hence also in $\sLoc{g}$). Circles denote the loci not in $\sLoc{g}$ (and hence not in $X$).
A squiggle denotes the recombination break interval $l$ considered in each case. The squares to the left and to the right of the squiggle respectively denote the loci in $\sLoc{\rec{g}{l}{-}}$ and $\sLoc{\rec{g}{l}{+}}$.
\subref{fig:case1} Case with $l < \min(X)$.
\subref{fig:case3} Case with $\min(X) \leq l < \max(X)$.
}
\label{fig:cases}
\end{figure}

\begin{enumerate}
\item $\bfmath{l < \min(X)}$: This case is illustrated in \fref{fig:case1}.\\  Note that $\sum_{\hpcond} \delta_{h', \rec{g}{l}{-}} = 0$ since $\sLoc{\rec{g}{l}{-}} \cap X = \varnothing$ and hence $\rec{g}{l}{-} \notin \hapSet{h}{X}$. 
Also, $\sum_{\hpcond} \delta_{h', \rec{g}{l}{+}} = 1$ since $g \succeq h$ and $\sLoc{\rec{g}{l}{+}} \supseteq X \supseteq \sLoc{h}$, and so $\rec{g}{l}{+} \succeq h$ and $\rec{g}{l}{+} \in \hapSet{h}{X}$. 
Hence, together with \eref{eq:sum1}, we conclude $\sum_{\hpcond} \delta_{h', \rec{g}{l}{-}} + \delta_{h', \rec{g}{l}{+}} - \delta_{h',g} = 0$. 
\item $\bfmath{\max(X) \leq l}$:  By a similar argument, we see that $\sum_{\hpcond} \delta_{h', \rec{g}{l}{-}} + \delta_{h', \rec{g}{l}{+}} - \delta_{h',g} = 0$.
\item $\bfmath{\min(X) \leq l < \max(X)}$:  This case is illustrated in \fref{fig:case3}.\\
Note that $\sum_{\hpcond} \delta_{h', \rec{g}{l}{-}} = 0$ since $\max(X) \notin \sLoc{\rec{g}{l}{-}}$ and hence $\rec{g}{l}{-} \notin \hapSet{h}{X}$. 
Similarly, $\sum_{\hpcond} \delta_{h', \rec{g}{l}{+}} = 0$ since $\min(X) \notin \sLoc{\rec{g}{l}{+}}$ and hence $\rec{g}{l}{+} \notin \hapSet{h}{X}$. 
Therefore, upon using \eref{eq:sum1}, we conclude $\sum_{\hpcond} \delta_{h', \rec{g}{l}{-}} + \delta_{h', \rec{g}{l}{+}} - \delta_{h',g} = -1$. 
\end{enumerate}

Partitioning the summation over $l\in\bpts{g}$ in the right hand side of expression \eref{eq:q1_tmp1} into the above three cases and noting that only the third case gives a non-zero value for the term $\sum_{\hpcond} \delta_{h', \rec{g}{l}{-}} + \delta_{h', \rec{g}{l}{+}} - \delta_{h',g}$, we obtain
\begin{align*}
\lefteqn{\sum_{g\in\cH} n_g \sum_{l\in B(g)} r_l \biggl[ q_1(\n) - q_1(\n - \e_g + \e_{\rec{g}{l}{-}} + \e_{\rec{g}{l}{+}} )\biggr]} & \notag \\    
    & =
    \sum_{h\in \cH\cup\{\un^\nLoc\}}  q_0(\margVec{\n} - \margVec{\e_h})  \sum_{\hXcond} \left\{  {(-1)^{|X-\sLoc{h}|} \over \rsum{X}} 
       \phantom{\sum_{l=\min(X)}^{\max(X)-1} }
 \right. \notag\\
    & \hspace{5cm}
\times \left.\sum_{g \in \hapSet{h}{X}} n_g \sum_{l=\min(X)}^{\max(X)-1} r_l  \left[ \left(\sum_{\hpcond} n_{h'} \right) - 1 \right] \right\}   \\
    & =
    \sum_{h\in \cH\cup\{\un^\nLoc\}}\left[\rule{0mm}{7mm}  q_0(\margVec{\n} - \margVec{\e_h}) \right. \\
 & \hspace{2cm}\times \sum_{\hXcond} (-1)^{|X-\sLoc{h}|} \Bigg( \sum_{h' \in \hapSet{h}{X}} n_h' \Bigg) \Bigg(\sum_{\hppcond} n_{h''} - 1 \Bigg) \left.\rule{0mm}{7mm}  \right].    
\end{align*}

\noindent
This is the expression on the right hand side of \eqref{eq:q1_simple_recurrence}, and thus we have shown that \eref{eq:thm} satisfies \eqref{eq:q1_simple_recurrence}. 
Hence, the proposed solution for $q_1(\n)$ in \thmref{thm:q1} is the unique solution to the recursion \eqref{eq:q1_simple_recurrence} and boundary conditions \eref{eq:q1_boundary}.
\hfill{}\qed


\section*{Acknowledgments}
We thank Paul Jenkins, Jack Kamm, and Matthias Steinr\"ucken for useful discussion and comments, and an anonymous reviewer for helpful comments.  This research is supported in part by an NIH grant R01-GM094402, an Alfred P. Sloan Research Fellowship, and a Packard Fellowship for Science and Engineering.

\bibliographystyle{apt}
\bibliography{master}

\end{document}

%% file: macros.tex
\newcommand{\bo}[1]{\ensuremath \boldsymbol{#1}}

\newcommand{\eref}[1]{(\ref{#1})}

\newcommand{\sref}[1]{Section~\ref{#1}}

\newcommand{\fref}[1]{Figure~\ref{#1}}
\newcommand{\cref}[1]{Chapter~\ref{#1}}

\newcommand{\lemmaref}[1]{Lemma~\ref{#1}}
\newcommand{\propref}[1]{Proposition~\ref{#1}}
\newcommand{\thmref}[1]{Theorem~\ref{#1}}

\def\bbR{\mathbb{R}}

\def\bfmath#1{\boldsymbol{#1}}

\def\bfn{{\bfmath{n}}}

\def\bfu{{\bfmath{u}}}

\def\cH{{\mathcal H}}

\def\tA{\theta_{\mbox{\scriptsize$A$}}}
\def\tB{\theta_{\mbox{\scriptsize $B$}}}

\def\P#1{P^{(#1)}}
\def\bfP#1{\bfmath{P}^{(#1)}}

\def\n{\bo{n}}
\def\m{\bo{m}}
\def\e{\bo{e}}
\def\ds{\ensuremath{\displaystyle}}
\def\un{*}

\def\hXcond{\substack{X: |X| \geq 2, \\ \sLoc{h}\subseteq X \subseteq [\nLoc]}}

\def\hpcond{h' \in \hapSet{h}{X}}
\def\hppcond{h'' \in \hapSet{h}{X}}

\def\rsum#1{r(#1)}

\def\sLoc#1{S(#1)} 
\def\nLoc{L}    
\def\bpts#1{B(#1)}
\def\marginal#1#2{#1^{(#2)}} 
\def\margVec#1{\sigma(#1)}
\usepackage{stmaryrd}
\DeclareMathOperator*{\compatible}{\varcurlywedge}
\def\st{:\;} 
\def\sub#1#2#3{M^{#3}_{#2}(#1)} 
\def\rec#1#2#3{R^{#3}_{#2}(#1)} 
\def\onelocq{p}

\def\hapSet#1#2{H(#1,#2)}       
\def\uvec#1#2{\bfu^{#1}_{#2}}         
\def\any{\cdot}
\renewcommand{\subset}{\subseteq}

\def\qed{$\Box$}